\documentclass{amsart}
\usepackage{amssymb}
\usepackage{amsmath}
\usepackage{amsfonts}
\usepackage{accents}

\newtheorem{theorem}{Theorem}
\theoremstyle{plain}

\newtheorem{definition}{Definition}
\newtheorem{example}{Example}

\newtheorem{lemma}{Lemma}

\newtheorem{proposition}{Proposition}

\numberwithin{equation}{section}
\input{tcilatex}

\begin{document}
\title{$\Delta $- convergence on time scale}
\author{M.Seyyit Seyyidoglu and N. \"{O}zkan TAN}
\address{Usak University,Faculty of Sciences and Arts,Department of
Mathematics, 1 Eylul Campus,64200,Usak, Turkey}
\email[M.Seyyit Seyyidoglu]{ seyyit.seyyidoglu@usak.edu.tr.\\
[N.\"{O}zkan TAN] nozkan.tan@usak.edu.tr.}
\subjclass[2000]{34N05}
\keywords{Time scale, Lebesgue $\Delta $-measure, Statistical convergence}

\begin{abstract}
In the present paper we will give some new notions, such as $\Delta $%
-convergence and $\Delta $-Cauchy, by using the $\Delta $-density and
investigate their relations. It is important to say that, the results
presented in this work generalize some of the results mentioned in the
theory of statistical convergence.
\end{abstract}

\maketitle

\section{Introduction and Background}

\bigskip In \cite{aa} Fast introduced an extension of the usual concept of
sequential limits which he called statistical convergence. In \cite{bb}
Schoenberg gave some basic properties of statistical convergence. In \cite%
{cc} Fridy introduced the concept of a statistically Cauchy sequence and
proved that it is equivalent to statistical convergence.

The theory of time scales was introduced by Hilger in his Ph. D. thesis in
1988, supervised by Auldbach \cite{ccc}. Measure theory on time scales was
first constructed by Guseinov \cite{dd}, then further studies were made by
Cabada-Vivero \cite{ee}, Rzezuchowski \cite{ff}. In \cite{gg} Deniz-Ufuktepe
define Lebesgue-Stieltjes $\Delta $ and $\bigtriangledown $-measures and by
using these measures \ they define an integral adapted to time scale,
specifically Lebesgue-Steltjes $\Delta $-integral.

A time scale $\mathbb{T}$ is an arbitrary nonempty closed subset of the real
numbers $\mathbb{R}$. The time scale $\mathbb{T}$ is a complete metric space
with the usual metric. We assume throughout that a time scale $\mathbb{T}$
has the topology that it inherits from the real numbers with the standart
topology.

For $t\in \mathbb{T}$ we define the \textit{forward jump operator} $\sigma :%
\mathbb{T\rightarrow T}$ by 
\begin{equation*}
\sigma \left( t\right) :=\inf \left \{ s\in \mathbb{T}\text{:~}s>t\right \} .
\end{equation*}

In this definition we put $\inf \{ \emptyset \}=\sup \mathbb{T}$.

For $a,b\in \mathbb{T}$ with $a\leq b$ we define the interval $\left[ a,b%
\right] $ in $\mathbb{T}$ by 
\begin{equation*}
\left[ a,b\right] =\left \{ t\in \mathbb{T}\text{: }a\leq t\leq b\right \} .
\end{equation*}%
Open intervals and half-open intervals etc. are defined accordingly.

Let $\mathbb{T}$ be time scale. Denote by $\mathcal{S}$ the family of all
left-closed and right-open intervals of $\mathbb{T}$ of the form $\left[
a,b\right) =\left \{ t\in \mathbb{T}\text{: }a\leq t<b\right \} $ with $%
a,b\in \mathbb{T}$ and $a\leq b$. The interval $\left[ a,a\right) $ is
understood as the empty set. $\mathcal{S}$ is a semiring of subsets of $%
\mathbb{T}$. Obviously the set function $m:\mathcal{S}\rightarrow \left[
0,\infty \right] $ defined by $m\left( \left[ a,b\right) \right) =b-a$ is a
countably additive measure. An outer measure $m^{\ast }:\mathcal{P}(\mathbb{T%
})\rightarrow \left[ 0,\infty \right] $ generated by $m$ is defined by 
\begin{equation*}
m^{\ast }(A)=\inf \left \{ \sum \limits_{n=1}^{\infty }m(A_{n}):(A_{n})\text{
is a sequence of }\mathcal{S}\text{ with }A\subset \bigcup
\limits_{n=1}^{\infty }A_{n}\right \} .
\end{equation*}%
If there is no sequence $(A_{n})$ of $\mathcal{S}$ such that $A\subset
\tbigcup \nolimits_{n=1}^{\infty }A_{n},$ \ then we let $m^{\ast }(A)=\infty
. $ We define the family $\mathcal{M}(m^{\ast })$ of all $m^{\ast }$%
-measurable subsets of $\mathbb{T}$, i.e. 
\begin{equation*}
\mathcal{M}(m^{\ast })=\left \{ E\subset \mathbb{T}:m^{\ast }(A)=m^{\ast
}(A\cap E)+m^{\ast }(A\cap E^{c})\text{ for all }A\subset \mathbb{T}\right
\} .
\end{equation*}%
The collection $\mathcal{M}(m^{\ast })$ of all $m^{\ast }$-measurable sets
is a $\sigma $-algebra and the restriction of $m^{\ast }$ to $\mathcal{M}%
(m^{\ast })$ which we denote by $\mu _{\Delta }$ is a countably additive
measure on $\mathcal{M}(m^{\ast })$. This measure $\mu _{\Delta }$ which is
Carath\'{e}odory extension of the set function $m$ associated with the
family $\mathcal{S}$ we call Lebesgue $\Delta $-measure on $\mathbb{T}$.

We call $f:\mathbb{T}\rightarrow \mathbb{R}$ is a measurable function, if $%
f^{-1}(\mathcal{O})\in \mathcal{M}(m^{\ast })$ for every open subset of $%
\mathcal{O}$ of $\mathbb{R}$.

\begin{theorem}
\label{em}(see \cite{dd}) For each $t_{0}$ in $\mathbb{T}-\left \{ \max 
\mathbb{T}\right \} $ the single point set $\left \{ t_{0}\right \} $ is $%
\Delta $-measurable and its $\Delta $-measure is given by%
\begin{equation*}
\mu _{\Delta }\left( \left \{ t_{0}\right \} \right) =\sigma (t_{0})-t_{0}.
\end{equation*}
\end{theorem}

\begin{theorem}
(see \cite{dd}) If $a,b\in \mathbb{T}$ and $a\leq b$, then%
\begin{equation*}
\begin{array}{cc}
\mu _{\Delta }\left( \left[ a,b\right) \right) =b-a, & \mu _{\Delta }\left(
\left( a,b\right) \right) =b-\sigma (a).%
\end{array}%
\end{equation*}%
If $a,b\in \mathbb{T}-\left \{ \max \mathbb{T}\right \} $ and $a\leq b$, then%
\begin{equation*}
\begin{array}{cc}
\mu _{\Delta }\left( \left( a,b\right] \right) =\sigma (b)-\sigma (a), & \mu
_{\Delta }\left( \left[ a,b\right] \right) =\sigma (b)-a.%
\end{array}%
\end{equation*}
\end{theorem}

It can easily seen from Theorem \ref{em} that, the measure of a subset of $%
\mathbb{N}$ is equal to its cardinality.

\section{$\Delta $-Density, $\Delta $-Convergence, $\Delta $-Cauchy}

It is well-known that the notions of statistical convergence and statistical
Cauchy are closely related with the density of the subset of $\mathbb{N}$.
In the present section first of all we will define the density of the subset
of the time scale. By using this definition, we will define $\Delta $%
-convergence and $\Delta $-Cauchy for a real valued function defined on the
time scale. After then we will show that these notions are equivalent.

Throughout this paper we consider the time scales, which are unbounded from
above and has a minimum point.

\bigskip Let $A$ be a $\Delta $-measurable subset of $\mathbb{T}$ and let $%
a=\min \mathbb{T}$. $\Delta $-density of $A$ in $\mathbb{T}$ (or briefly $%
\Delta $-density of $A)$ is defined by%
\begin{equation*}
\delta _{\Delta }(A)=\lim_{t\rightarrow \infty }\frac{\mu _{\Delta }(A(t))}{%
\sigma (t)-a}
\end{equation*}%
(if this limit exists) where%
\begin{equation}
A(t)=\left \{ s\in A:s\leq t\right \}  \label{lem}
\end{equation}%
and $t\in \mathbb{T}.$

From the identity $A(t)=A\cap \left[ a,t\right] $, measurability of $A$
implies measurability of $A(t)$.

If $f:\mathbb{T}\rightarrow \mathbb{R}$ is a function such that $f(t)$
satisfies property $P$ for all $t$ except a set of $\Delta $-density zero,
then we say that $f(t)$ satisfies $P$ for "$\Delta $-almost all $t$", and we
abbreviate this by $"\Delta $-$a.a.$ $t.".$ Let 
\begin{eqnarray*}
\mathcal{M}_{d} &:&=\left \{ A\in \mathbb{T}:\Delta \text{-density of }A%
\text{ is exists in }\mathbb{T~}\right \} , \\
\mathcal{M}_{d}^{0} &:&=\left \{ A\in \mathbb{T}:\delta _{\Delta
}(A)=0\right \} .
\end{eqnarray*}

\begin{lemma}
\label{qq}i) If $A,B\in $ $\mathcal{M}_{d}$ and $A\subset B,$ then $\delta
_{\Delta }(A)\leq \delta _{\Delta }(B),$

ii) If $A\in \mathcal{M}_{d},$ then $0\leq \delta _{\Delta }(A)\leq 1,$

iii) $\mathbb{T\in }\mathcal{M}_{d}$ and $\delta _{\Delta }(\mathbb{T})=1,$

iv) If $A\in \mathcal{M}_{d},$ then $A^{c}\in \mathcal{M}_{d}$ and $\delta
_{\Delta }(A)+\delta _{\Delta }(A^{c})=1,$

v) If $A,B\in $ $\mathcal{M}_{d}$ and $A\subset B,$ then $B-A\in \mathcal{M}%
_{d}$ and $\delta _{\Delta }(B-A)=\delta _{\Delta }(B)-\delta _{\Delta }(A),$

vi) If $\ A_{1},A_{2},...,A_{n}$ is a mutually disjoint sequence in $%
\mathcal{M}_{d}$ $,$ then $\tbigcup \nolimits_{k=1}^{n}A_{k}\in \mathcal{M}%
_{d}$ and 
\begin{equation*}
\delta _{\Delta }\left( \bigcup \limits_{k=1}^{n}A_{k}\right) =\sum
\limits_{k=1}^{n}\delta _{\Delta }(A_{k}),
\end{equation*}

vii) If $A_{1},A_{2},...,A_{n}\in \mathcal{M}_{d}$ and $\tbigcup%
\nolimits_{k=1}^{n}A_{k}\in \mathcal{M}_{d},$ then%
\begin{equation*}
\delta _{\Delta }\left( \bigcup \limits_{k=1}^{n}A_{k}\right) \leq \sum
\limits_{k=1}^{n}\delta _{\Delta }(A_{k}),
\end{equation*}

viii) If $A$ is measurable set and $B\in \mathcal{M}_{d}^{0}$ with $A\subset
B$ then $A\in \mathcal{M}_{d}^{0},$

ix) If $A_{1},A_{2},...,A_{n}\in \mathcal{M}_{d}^{0},$ then $\tbigcup
\nolimits_{k=1}^{n}A_{k},~\tbigcap \nolimits_{k=1}^{n}A_{k}\in \mathcal{M}%
_{d}^{0},$

x) Every bounded measurable subset of $\mathbb{T}$ belongs to $\mathcal{M}%
_{d}^{0},$

xi)If $A\in \mathcal{M}_{d}^{0}$ and $B\in \mathcal{M}_{d}$ then $\delta
_{\Delta }(A\cup B)=\delta _{\Delta }(B).$
\end{lemma}

\begin{proof}
i) Let $A,B\in $ $\mathcal{M}_{d}$ and $A\subset B$. \ Clearly $A(t)\subset
B(t)$ and since $\mu _{\Delta }$ is a measure function, one has $\mu
_{\Delta }(A)\leq \mu _{\Delta }(B)$. Thus we have;%
\begin{equation*}
\delta _{\Delta }(A)=\lim_{t\rightarrow \infty }\frac{\mu _{\Delta }(A(t))}{%
\sigma (t)-a}\leq \lim_{t\rightarrow \infty }\frac{\mu _{\Delta }(B(t))}{%
\sigma (t)-a}=\delta _{\Delta }(B)
\end{equation*}

ii) Note that $A(t)\subset \lbrack a,t].$ The required inequalities follows
from the following inequalities;%
\begin{equation*}
0\leq \frac{\mu _{\Delta }(A(t))}{\sigma (t)-a}\leq \frac{\mu _{\Delta
}([a,t])}{\sigma (t)-a}=\frac{\sigma (t)-a}{\sigma (t)-a}=1.
\end{equation*}%
iii) It is clear that $\mathbb{T}$ is measurable. The $\Delta $-density of $%
\mathbb{T}$ is obtained from the following equalities;%
\begin{equation*}
\delta _{\Delta }(\mathbb{T})=\lim_{t\rightarrow \infty }\frac{\mu _{\Delta
}(\mathbb{T}(t))}{\sigma (t)-a}=\lim_{t\rightarrow \infty }\frac{\mu
_{\Delta }(\mathbb{[}a,t])}{\sigma (t)-a}=1
\end{equation*}%
iv) Since $A$ is measurable so is $A^{c},$ namely $A^{c}(t)=\left \{ s\leq
t:s\in A^{c}\right \} $ is measurable. On the other hand $A(t)\cup
A^{c}(t)=[a,t]$ and%
\begin{equation*}
1=\frac{\mu _{\Delta }(\mathbb{[}a,t])}{\sigma (t)-a}=\frac{\mu _{\Delta
}(A(t))}{\sigma (t)-a}+\frac{\mu _{\Delta }(A^{c}(t))}{\sigma (t)-a}
\end{equation*}%
imply that the required statement holds as $t\rightarrow \infty .$

v) Since $A$ and $B$ are measurable so are $B-A.$The statement can be easily
shown by considering $A(t)\cup (B-A)(t)\cup B^{c}(t)=[a,t]$.

vi) Since the $\Delta $-density of for each subset $A_{k}$ exists, one can
write 
\begin{eqnarray*}
\sum \limits_{k=1}^{n}\delta _{\Delta }(A_{k}) &=&\sum
\limits_{k=1}^{n}\lim_{t\rightarrow \infty }\frac{\mu _{\Delta }(A_{k}(t))}{%
\sigma (t)-a} \\
&=&\lim_{t\rightarrow \infty }\sum \limits_{k=1}^{n}\frac{\mu _{\Delta
}(A_{k}(t))}{\sigma (t)-a} \\
&=&\lim_{t\rightarrow \infty }\frac{\mu _{\Delta }\left( \bigcup
\limits_{k=1}^{n}A_{k}(t)\right) }{\sigma (t)-a} \\
&=&\lim_{t\rightarrow \infty }\frac{\mu _{\Delta }\left( \left( \bigcup
\limits_{k=1}^{n}A_{k}\right) (t)\right) }{\sigma (t)-a} \\
&=&\delta _{\Delta }\left( \bigcup \limits_{k=1}^{n}A_{k}\right)
\end{eqnarray*}

vii) The proof is similar to that of the previous proof.

viii) It can be easily seen from (i).

ix) Considering $\delta _{\Delta }(A_{k})=0$ for $k=1,2,...,n$ and (vii),
one can obtain $\tbigcup \nolimits_{k=1}^{n}A_{k}\in \mathcal{M}_{d}^{0}$.
And $\tbigcap \nolimits_{k=1}^{n}A_{k}\in \mathcal{M}_{d}^{0}$ can be obtain
from (viii).

x) Let $A$ be a bounded set. For a sufficiently large $M\in \mathbb{T}$ we
can write $A\subset \lbrack a,M]$. Then one has%
\begin{equation*}
0\leq \frac{\mu _{\Delta }(A(t))}{\sigma (t)-a}\leq \frac{\mu _{\Delta
}([a,M])}{\sigma (t)-a}\rightarrow 0~\ ~~\ ~~~(t\rightarrow \infty ),
\end{equation*}
this implies that $\delta _{\Delta }(A)=0$.

xi) (i) and (vii) yields $\delta _{\Delta }(B)\leq $ $\delta _{\Delta
}(A\cup B)\leq $ $\delta _{\Delta }(A)+$ $\delta _{\Delta }(B)=$ $\delta
_{\Delta }(B)$. This completes the proof.
\end{proof}

It is clear that the familiy $\mathcal{M}_{d}^{0}$ is a ring of subsets of $%
\mathbb{T}$. According to (iv), the $\Delta $-density of the complement of a
subset whose $\Delta $-density is 0 is equal to $1$, $\mathcal{M}_{d}^{0}$
is not closed under the operation complement. So it is not an algebra. Note
that the $\Delta $-density of a subset of $\mathbb{N}$ is equal to its
natural density.

\begin{example}
Let $\mathbb{T}=[0,\infty ),$ $l$ and $r$ be arbitrary two positive real
numbers.Let also $A=\tbigcup \nolimits_{n=0}^{\infty }A_{n},$ where $%
A_{n}=[nl+nr,(n+1)l+nr].$ According to Lemma \ref{qq}-(x) for each $A_{n}$
is bounded and so $\delta _{\Delta }(A_{n})=0$. In addition,let $A(t)$
defined\ as in (\ref{lem}), we have%
\begin{equation*}
\mu _{\Delta }\left( A\left( t\right) \right) =\left \{ 
\begin{array}{cc}
t-nr~~\  \ , & nl+nr\leq t\leq (n+1)l+nr \\ 
(n+1)l\  \  \ , & (n+1)l+nr\leq t\leq (n+1)l+(n+1)r%
\end{array}%
\right.
\end{equation*}%
$(n=0,1,2,...)$ and hence%
\begin{equation*}
\delta _{\Delta }(A)=\lim_{t\rightarrow \infty }\frac{\mu _{\Delta }(A(t))}{%
\sigma (t)-a}=\lim_{t\rightarrow \infty }\frac{\mu _{\Delta }(A(t))}{t-0}=%
\frac{l}{l+r}.
\end{equation*}%
Note that, since%
\begin{equation*}
\frac{l}{l+r}=\delta _{\Delta }\left( A\right) >\sum \limits_{n=0}^{\infty
}\delta _{\Delta }\left( A_{n}\right) =0
\end{equation*}%
$\delta _{\Delta }$ does not define a measure.
\end{example}

\begin{example}
Let $\mathbb{T}=\mathbb{N}$. The $\Delta $-density of $A=\{2,4,6,...\}$ in $%
\mathbb{N}$ is given by 
\begin{equation*}
\delta _{\Delta }(A)=\lim_{t\rightarrow \infty }\frac{\mu _{\Delta }(A(t))}{%
\sigma (t)-a}=\lim_{t\rightarrow \infty }\frac{\left[ t/2\right] }{t}=\frac{1%
}{2}.
\end{equation*}
\end{example}

\begin{definition}
($\Delta $-convergence) The function $f:\mathbb{T}\rightarrow \mathbb{R}$ is 
$\Delta $-convergent to the number $L$ provided that for each $\varepsilon
>0,$ there exists $K_{\varepsilon }\subset \mathbb{T}$ such that $\delta
_{\Delta }(K_{\varepsilon })=1$ and $\left \vert f(t)-L\right \vert
<\varepsilon $ holds for all $t\in K_{\varepsilon }$.
\end{definition}

We will use notation $\Delta $-$\lim_{t\rightarrow \infty }f(t)=L$ .

\begin{definition}
($\Delta $-Cauchy) The function $f:\mathbb{T}\rightarrow \mathbb{R}$ is $%
\Delta $-Cauchy provided that for each $\varepsilon >0,$ there exists $%
K_{\varepsilon }\subset \mathbb{T}$ and $t_{0}\in \mathbb{T}$ such that $%
\delta _{\Delta }(K_{\varepsilon })=1$ and $\left \vert f(t)-f(t_{0})\right
\vert <\varepsilon $ holds for all $t\in K_{\varepsilon }$ .
\end{definition}

\begin{proposition}
Let $f:\mathbb{T}\rightarrow \mathbb{R}$ be a measurable function. $\Delta $-%
$\lim_{t\rightarrow \infty }f(t)=L$ if and only if for each $\varepsilon >0, 
$ $\  \delta _{\Delta }\left( \left \{ t\in \mathbb{T}:\left \vert
f(t)-L\right \vert \geq \varepsilon \right \} \right) =0.$
\end{proposition}

\begin{proof}
Let $\Delta $-$\lim_{t\rightarrow \infty }f(t)=L$ and $\varepsilon >0$ be
given. In this case there exists a subset $K_{\varepsilon }\subset \mathbb{T}
$ such that $\delta _{\Delta }(K_{\varepsilon })=1$ and $\left \vert
f(t)-L\right \vert <\varepsilon $ holds for all $t\in K_{\varepsilon }$.
Since $K_{\varepsilon }\subset \left \{ t\in \mathbb{T}:\left \vert
f(t)-L\right \vert <\varepsilon \right \} $ we obtain $\delta _{\Delta
}(\left \{ t\in \mathbb{T}:\left \vert f(t)-L\right \vert <\varepsilon
\right \} )=1$. Hence we get $\delta _{\Delta }(\left \{ t\in \mathbb{T}%
:\left \vert f(t)-L\right \vert \geq \varepsilon \right \} )=0$.

Other case of the proof is straightforward.
\end{proof}

\begin{proposition}
Let $f:\mathbb{T}\rightarrow \mathbb{R}$ be a measurable function. $f$ is $%
\Delta $-Cauchy if and only if for each $\varepsilon >0,$ there exists $%
t_{0}\in \mathbb{T}$ $\ $such that $\delta _{\Delta }\left( \left \{ t\in 
\mathbb{T}:\left \vert f(t)-f(t_{0})\right \vert \geq \varepsilon \right \}
\right) =0.$
\end{proposition}

\begin{example}
Let $\mathbb{I}_{[0,\infty )}$ be irrational numbers and $\mathbb{Q}%
_{[0,\infty )}$ be rational numbers in $[0,\infty )$. Let consider the
function $f:\mathbb{T}=[0,\infty )\rightarrow \mathbb{R}$ defined as
following,%
\begin{equation*}
f(t)=\left \{ 
\begin{array}{cc}
0 & t\in \mathbb{Q}_{[0,\infty )} \\ 
1 & t\in \mathbb{I}_{[0,\infty )}%
\end{array}%
\right. .
\end{equation*}%
Since $\mu _{\Delta }\left( \mathbb{Q}_{[0,\infty )}\right) =0,$ the density
of the subset $\mathbb{Q}_{[0,\infty )}$ in the time scale $\mathbb{T}$ is
zero. This implies that $\delta _{\Delta }(\mathbb{I}_{[0,\infty )})=1$. So,
for each $\varepsilon >0$ and for all $t\in \mathbb{I}_{[0,\infty )}$ one
has $0=\left \vert f(t)-1\right \vert <\varepsilon $ and as a corollary, we
get $\Delta $-$\lim_{t\rightarrow \infty }f(t)=1.$
\end{example}

\begin{proposition}
The $\Delta $-limit of a function $f:\mathbb{T}\rightarrow \mathbb{R}$ is
unique.
\end{proposition}

\begin{proof}
Let $\Delta $-$\lim_{t\rightarrow \infty }f(t)=L_{1}$ and $\Delta $-$%
\lim_{t\rightarrow \infty }f(t)=L_{2}$. Let $\varepsilon >0$ be given. Then
there exist subsets $K_{1},$ $K_{2}$ $\subset \mathbb{T}$ such that for
every $t\in K_{1}$ with $\left \vert f(t)-L_{1}\right \vert <\varepsilon /2$
and for every $t\in K_{2}$ with $\left \vert f(t)-L_{1}\right \vert
<\varepsilon /2,$ where $\delta _{\Delta }(K_{1})=1$ and $\delta _{\Delta
}(K_{2})=1.$ From Lemma \ref{qq}-(\textit{iv}) we have $K_{1}\cap K_{2}\neq
\varnothing .$ Thus for every $t\in K_{1}\cap K_{2}$ one has 
\begin{equation*}
\left \vert L_{1}-L_{2}\right \vert \leq \left \vert f(t)-L_{1}\right \vert
+\left \vert f(t)-L_{2}\right \vert <\varepsilon .
\end{equation*}
Thus $L_{1}=L_{2}.$
\end{proof}

\begin{proposition}
If  $f,g:\mathbb{T}\rightarrow \mathbb{R}$ with $\  \Delta $-$%
\lim_{t\rightarrow \infty }f(t)=L_{1}$ and $\Delta $-$\lim_{t\rightarrow
\infty }g(t)=L_{2}$, then the following statements hold:
\end{proposition}

i) $\Delta $-$\lim_{t\rightarrow \infty }\left( f(t)+g(t)\right)
=L_{1}+L_{2},$

ii) $\Delta $-$\lim_{t\rightarrow \infty }\left( cf(t)\right) =cL_{1}.(c\in 
\mathbb{R}).$

\begin{proposition}
If $f:\mathbb{T}\rightarrow \mathbb{R}$ with $\lim \nolimits_{t\rightarrow
\infty }f(t)=L$ then $\Delta $-$\lim_{t\rightarrow \infty }f(t)=L.$
\end{proposition}

\begin{proof}
Let $\lim \nolimits_{t\rightarrow \infty }f(t)=L$. In this case for a given $%
\varepsilon >0$ we can find a $t_{0}\in \mathbb{T}$ such that $\left \vert
f(t)-L\right \vert <\varepsilon $ holds for every $t>t_{0}$. The set $%
K_{\varepsilon }=\left \{ t\in \mathbb{T}:t>t_{0}\right \} $ is measurable
and from Lemma \ref{qq} (iv) and (\textit{x}) one has $\delta _{\Delta
}(K_{\varepsilon })=1.$ By the definition of $\Delta $-convergence we get\ $%
\Delta $-$\lim_{t\rightarrow \infty }f(t)=L.$
\end{proof}

\begin{theorem}
\label{ana} Let $f:\mathbb{T}\rightarrow \mathbb{R}$ be a measurable
function. The following statements are equivalent:

i) $f$ is $\Delta $-convergent,

ii) $f$ is $\Delta $-Cauchy,

iii) There exists a measurable and convergent function $g:\mathbb{T}%
\rightarrow \mathbb{R}$ such that $f(t)=g(t)~$for $\Delta $-$a.a.~t.$
\end{theorem}

\begin{proof}
(i)$\Rightarrow $(ii) : Let $\Delta $-$\lim_{t\rightarrow \infty }f(t)=L$
and $\varepsilon >0$ be given. Then $\left \vert f(t)-L\right \vert
<\varepsilon /2$ holds for $\Delta $-$a.a.~t$. We can choose $t_{0}\in 
\mathbb{T}$ such that $\left \vert f(t_{0})-L\right \vert <\varepsilon /2$
holds. So,%
\begin{equation*}
\left \vert f(t)-f(t_{0})\right \vert \leq \left \vert f(t)-L\right \vert
+\left \vert f(t_{0})-L\right \vert <\varepsilon ~~~~\Delta \text{-}a.a.~t.
\end{equation*}%
This shows that $f$ satisfies the property of $\Delta $-Cauchy.

(ii)$\Rightarrow $(iii) : We can choose an element $t_{1}\in \mathbb{T}$. We
can define an interval $I=[f(t_{1})-1,f(t_{1})+1]$ which contains $f(t)$ for 
$\Delta $-$a.a.~t.$ By the same method we can choose an element $t^{\ast
}\in \mathbb{T}$ and define an interval $I^{\prime }=[f(t^{\ast
})-1/2,f(t^{\ast })+1/2]$ which contains $f(t)$ for $\Delta $-$a.a.~t$. We
can write%
\begin{equation*}
\left \{ t\leq s:f(t)\notin I\cap I^{\prime }\right \} =\left \{ t\leq
s:f(t)\notin I\right \} \cup \left \{ t\leq s:f(t)\notin I^{\prime }\right \} .
\end{equation*}%
Since $f$ is a measurable function, the two terms appear on the
righthand-side of the last equality are also measurable. By using Lemma \ref%
{qq}-(vii) we obtain%
\begin{eqnarray*}
\delta _{\Delta }\left( \left \{ t\in \mathbb{T}:f(t)\notin I\cap I^{\prime
}\right \} \right)  &\leq &\delta _{\Delta }\left( \left \{ t\in \mathbb{T}%
:f(t)\notin I\right \} \right) +\delta _{\Delta }\left( \left \{ t\in \mathbb{T%
}:f(t)\notin I^{\prime }\right \} \right)  \\
&=&0.
\end{eqnarray*}%
So,  $f(t)$ is in the closed interval $I_{1}=I\cap I^{\prime }$ for $\Delta $%
-$a.a.~t.$ It is clear that the lengt of the interval $I_{1}$ is less than
or equal $1$. Now we can choose $t_{2}\in \mathbb{T}$ with\ $I^{\prime
\prime }=[f(t_{2})-1/4,f(t_{2})+1/4]$ contains $f(t)$ for $\Delta $-$a.a.~t.$
Undoubtly the closed interval $I_{2}=I_{1}\cap I^{\prime \prime }$ contains $%
f(t)$ for $\Delta $-$a.a.~t$ and the lengt of the interval $I_{2}$ is less
than or equal $1/2$. With the same procedure, for each $m$ we can obtain a
sequence of closed intervals $(I_{m})_{m=1}^{\infty }$such that $%
I_{m+1}\subset I_{m}$ and the lengt of each interval $I_{m}$ is less than or
equal $2^{1-m}$. Moreover $f(t)\in I_{m}$ for $\Delta $-$a.a.~t.$ From the
properties of the intersection of closed intervals there exists a real
number $\lambda $ such that $\tbigcap \nolimits_{m=1}^{\infty
}I_{m}=\{ \lambda \}.$ Since the $\Delta $-density of the set on which $%
f(t)\notin I_{m}$ is equal to zero , we can find an increasing sequence $%
(T_{m})_{m=1}^{\infty }$in $\mathbb{T}$ such that,%
\begin{equation}
\frac{\mu _{\Delta }\left( \left \{ t\leq s:f(t)\notin I_{m}\right \} \right) 
}{\sigma (s)-a}<\frac{1}{m}~\text{\ }~~~,\text{ \  \ }s>T_{m}.  \label{estz}
\end{equation}%
Here $a=\min \mathbb{T}$. Let consider the function $g:\mathbb{T}\rightarrow 
\mathbb{R}$ defined as following: 
\begin{equation*}
g(t)=\left \{ 
\begin{array}{ll}
\lambda  & ,~~~~~T_{m}<t\leq T_{m+1}~~\text{and \ }f(t)\notin I_{m} \\ 
f(t) & ,~~~~~\text{otherwise.}%
\end{array}%
\right. 
\end{equation*}%
It is clear that $g$ is a measurable function and $\lim \nolimits_{t%
\rightarrow \infty }g(t)=\lambda $. Indeed, for $t>T_{m}$ either $%
g(t)=\lambda $ or $g(t)=f(t)\in I_{m}.$ In this case $\left \vert
g(t)-\lambda \right \vert \leq 2^{1-m}$ holds.

Finally we shall show that $f(t)=g(t)$ for $\Delta $-$a.a.~t.$ For this
purpose consider, 
\begin{equation*}
\left \{ t\leq s:f(t)\neq g(t)\right \} \subset \left \{ t\leq s:f(t)\notin
I_{m}\right \} ,
\end{equation*}%
$T_{m}<s\leq T_{m+1}.$ Thus we get from (\ref{estz})%
\begin{equation*}
\frac{\mu _{\Delta }\left( \left \{ t\leq s:f(t)\neq g(t)\right \} \right) }{%
\sigma (s)-a}\leq \frac{\mu _{\Delta }\left( \left \{ t\leq s:f(t)\notin
I_{m}\right \} \right) }{\sigma (s)-a}<\frac{1}{m}
\end{equation*}%
this yielding $\delta _{\Delta }(\left \{ t\in \mathbb{T}:f(t)\neq
g(t)\right \} =0,$ that is, $f(t)=g(t)~$for $\Delta $-$a.a.~t.$

(iii)$\Rightarrow $(i) : Let $f(t)=g(t)~$for $\Delta $-$a.a.~t.$ and $\lim
\nolimits_{t\rightarrow \infty }g(t)=L.$ For a given $\varepsilon >0$ we
have,%
\begin{equation*}
\left \{ t\in \mathbb{T}:\left \vert f(t)-L\right \vert \geq \varepsilon
\right \} \subset \left \{ t\in \mathbb{T}:f(t)\neq g(t)\right \} \cup \left
\{ t\in \mathbb{T}:\left \vert g(t)-L\right \vert \geq \varepsilon \right \}
.
\end{equation*}%
Since $\lim \nolimits_{t\rightarrow \infty }g(t)=L$ , the second set appears
on the righthand-side of the above inclusion relation is bounded and thus $%
\delta _{\Delta }(\left \{ t\in \mathbb{T}:\left \vert g(t)-L\right \vert
\geq \varepsilon \right \} )=0$. In addition $f(t)=g(t)~$for $\Delta $-$%
a.a.~t$ yields $\delta _{\Delta }(\left \{ t\in \mathbb{T}:f(t)\neq
g(t)\right \} )=0$. In conclusion $\delta _{\Delta }(\left \{ t\in \mathbb{T}%
:\left \vert f(t)-L\right \vert \geq \varepsilon \right \} )=0,$ namely $%
\Delta $-$\lim_{t\rightarrow \infty }f(t)=L$.
\end{proof}

\end{document}